\documentclass[11pt]{amsart}
\usepackage{amscd,amsmath,amssymb,amsfonts,verbatim, amsthm}
\usepackage[cmtip, all]{xy}
\usepackage{comment}
\usepackage{MnSymbol}

\usepackage{amsaddr}

\usepackage{pdfpages}
\usepackage [english]{babel}
\usepackage [autostyle, english = american]{csquotes}

{\large }

\theoremstyle{plain}
\newtheorem{thm}{Theorem}[section]

\newtheorem{thmA}{Theorem}

\newtheorem*{thm*}{Theorem}

\newtheorem{prop}[thm]{Proposition}
\newtheorem{lem}[thm]{Lemma}

\newtheorem*{cor*}{Corollary}

\theoremstyle{definition}

\newtheorem{defn}[thm]{Definition}
\theoremstyle{remark}
\newtheorem{remk}[thm]{Remark}
\newtheorem{remks}[thm]{Remarks}

\newtheorem{exm}[thm]{Example}
\newtheorem{exms}[thm]{Examples}

\numberwithin{equation}{section}

{\hfill$\square$\end{defn}}
{\hfill$\square$\end{remk}}
{\hfill$\square$\end{remks}}
{\hfill$\square$\end{exm}}
{\hfill$\square$\end{exms}}%

\newcommand{\thmref}{Theorem~\ref}
\newcommand{\propref}{Proposition~\ref}

\newcommand{\lemref}{Lemma~\ref}

\newcommand{\remref}{Remark~\ref}

\newcommand{\secref}{Section~\ref}

\newcommand{\sE}{{\mathcal E}}
\newcommand{\sF}{{\mathcal F}}

\newcommand{\sM}{{\mathcal M}}

\newcommand{\sO}{{\mathcal O}}

\newcommand{\A}{{\mathbb A}}

\newcommand{\C}{{\mathbb C}}

\renewcommand{\P}{{\mathbb P}}
\newcommand{\Q}{{\mathbb Q}}

\newcommand{\Z}{{\mathbb Z}}

\newcommand{\CH}{{\rm CH}}

\newcommand{\surj}{\twoheadrightarrow}
\newcommand{\inj}{\hookrightarrow}

\newcommand{\Spec}{{\rm Spec \,}}

\newcommand{\id}{{\operatorname{id}}}

\newcommand{\Sch}{{\operatorname{\mathbf{Sch}}}}

\newcommand{\op}{{\text{\rm op}}}

\renewcommand{\>}{\rangle}

\newcommand{\Sm}{{\mathbf{Sm}}}
\newcommand{\SmProj}{{\mathbf{SmProj}}}

\newcommand{\Ab}{{\mathbf{Ab}}}

\newcommand{\ds}{{/\kern-3pt/}}

\newcommand{\Proj}{{\operatorname{Proj}}}

\renewcommand{\dim}{\text{\rm dim}}

\newcommand{\tuborg}{\left\{\begin{array}{ll}}
\newcommand{\sluttuborg}{\end{array}\right.}

\newcommand{\pr}{{\rm pr}}

\newcounter{elno}

\newcounter{elno-abc}   

\newcounter{elno-abc-prime}

\newcommand{\TCH}{{\rm TCH}}

\newcommand{\Mor}{\textnormal{Mor}}

 \newcommand{\Corr}{\textnormal{Corr}}
 \newcommand{\Co}{{\bf Corr}}
 \newcommand{\ProjSm}{{\mathbf{ProjSm}}}

\begin{document}


\title{Motivic invariants of symmetric powers of curves}

\author{Rahul Gupta}

\address{Universität Regensburg}

\address{
 Fakultät für Mathematik,\\
Universität Regensburg,
 93049 Regensburg, Germany\\
 rahul.gupta@mathematik.uni-regensburg.de }






\maketitle

\begin{quote}\emph{Abstract.}  
We study the structure of various invariants of the symmetric powers of a smooth projective  curve in terms of that of the Jacobian of the curve. We generalise the results of  Macdonald and Collino to various invariants including the Weil-cohomology theory, the higher Chow groups, the additive higher Chow groups and the rational $K$-groups. 
\end{quote}

\vskip .5pc

\begin{quote} \emph{Keywords.}
Algebraic cycles, correspondences, Chow motives, Jacobian.
\end{quote}

\vskip .5pc

\begin{quote} \emph{Mathematics Subject Classification $[2020]$.}
Primary 14C25; Secondary 19E08, 14H99
\end{quote}

\setcounter{tocdepth}{1}
\tableofcontents

\section{Introduction} \label{sec:intro}
%
%

The aim of this paper is to study a certain class of invariants of the symmetric powers of 
a smooth projective curve. 
Note that the Jacobian of a smooth projective curve is an abelian variety and in general, we know much more about invariants of an abelian variety than those of a general scheme.  MacDonald \cite{Mac} studied the singular cohomology ring of the symmetric powers in terms of the cohomology ring of the Jacobian of the curve  and Collino \cite{Co} studied the same for the Chow ring. 
Our aim is to study similar kinds of relations for a class of invariants which includes
 the higher Chow groups and the rational $K$-theory. We also prove similar results for 
 non-homotopic invariants like  the additive higher Chow groups and 
 the higher Chow groups with 
 modulus. 
Moreover, we show that the results of Macdonald and Collino ``are motivic'', i.e., they can be proven by a uniform method. 
We also prove similar results for any Weil cohomology theory. We now introduce some notations and discuss  the problem in more detail. 

Let $C$ be a smooth projective curve over an algebraically closed field $k$ and let $C(n)$ denote the $n$-th symmetric power ${\rm Sym}^n(C)$ of $C$. Then $C(n)$ is an $n$-dimensional smooth scheme 
 and the closed points of  $C(n)$ can be identified with the set of effective degree-$n$ $0$-cycles on $C$, or equivalently with the set of unordered $n$-tuples of closed points of $C$. 
 Let $J(C)$ denote the Jacobian of $C$. Recall that the set of closed points of $J(C)$ is the set of the degree-zero $0$-cycles on $C$.  A closed point $p \in C$ defines the morphism 
 $\pi_n\colon C(n) \to J(C)$ and
 the closed embedding $\iota_{m,n}\colon C(m) \to C(n)$ for $m\leq n$, so that  
 $\pi_n((x_1, \dots, x_n)) = \sum_i [x_i] - n[p]$ 
 and 
 $\iota_{m,n} ((x_1, \dots, x_m)) = (x_1, \dots, x_m, p, \dots, p)$. 
It is known by \cite{Sh-63} that there exist coherent sheaves $\sF_n$ on $J(C)$ such that $C(n)= \Proj(\sF_n)$ as schemes over $J(C)$. 
 If $g$ denotes the genus of $C$ and $n \geq 2g-1$, then $\sF_n$ is locally free and hence $C(n)$ is a projective bundle over $J(C)$. In this case, we denote $\sF_n$ by $\sE_n$.

 Let $G \colon \ProjSm_k^{\op} \to {\bf Rings}$ be a functor on smooth projective schemes which  satisfies the projective bundle formula. For example, it can be the higher Chow groups $\CH^*(-, \bullet)$, the $K$-theory ring or the singular cohomology ring (in case $k=\C$).  Using 
  the projective bundle formula, it is easy to describe  $G(C(n))$ as an algebra over $G(J(C))$ for all $n\geq 2g-1$. 
 MacDonald \cite{Mac} studied the singular cohomology ring $H^*(C(n), \Z)$ of $C(n)$ as an algebra over $H^*(J(C), \Z)$ 
 for  $n < 2g-1$ while Collino \cite{Co} studied the Chow ring of $C(n)$ as an algebra over that of $J(C)$ for  $n< 2g-1$. MacDonald used Kunneth formula while Collino used the localization exact sequence for the Chow rings along with a moving lemma. 
 The main aim of this paper is to study the structure of $G(C(n))$
  as an algebra over $G(J(C))$ for $n< 2g-1$ and  for a general functor $G$ 
  which is close to a good cohomology theory. 
 For precise conditions on $G$, see Theorems~\ref{thm:thm-A}  and~\ref{thm:thm-B}
 in Sections~\ref{sec:thm-A} and~\ref{sec:thm-B}, respectively.
 
 We let 
 $u_i = c_i(\sE_n) \in \CH^i(J(C))$ denote the $i$-th Chern class of $\sE_n$ for $n \geq 2g-1$ and $0\leq i \leq n-g+1$. 
  By  \cite{Mat}, it follows that the elements $u_i$ do not depend on $n$ 
  (see \secref{sec:Symm-P-Corr-J} for details).
We let $\alpha = \sum_{0 \leq i \leq g} u_i z^{g-i} \in \CH^*(J(C))[z]$. 
 With these notations, we obtain the following applications of 
 Theorems~\ref{thm:thm-A} and~\ref{thm:thm-B}.

 \vskip .2cm
 
 \subsection{Classical Weil Cohomology ring}
 
 Let $k$ be an algebraically closed field and let $K$ be a field of characteristic zero. Let $H^*\colon  \ProjSm_k \to {\bf GrAlg}_K$ be a classical Weil cohomology theory (\cite{Kle68}).
 Recall that for $X \in \ProjSm_k$, there exists a cycle class map ${\rm cl}\colon  \CH^*(X) \to H^{2*}(X)$ which is a homomorphism of graded rings. As an application to  \thmref{thm:thm-A}, we prove the following.

 \begin{thm} \label{thm:WCT-Structure}
   Let $v_i = {\rm cl}(u_i) \in H^{2i}(J(C))$, $\beta = \sum_{0 \leq i \leq g} v_i z^{g-i} \in H^*(J(C))[z]$
and 
\[
I_n  = \begin{cases}  ((\beta): z^{2g-1-n})& \textnormal{ if } n <  2g-1,\\
(\beta \cdot z^{n-2g+1})&\textnormal{ if } n \geq  2g-1.
\end{cases}
\]
For every $n\geq 0$, we then have the short exact sequence 
\begin{equation} \label{eqn:thm-for-WCT*-1}
0 \to I_n \to H^*(J(C))[z] \xrightarrow{\psi_n} H^*(C(n)) \to  0,
\end{equation}
where $\psi_n\colon  H^*(J(C))[z] \to H^*(C(n))$ is a graded ring homomorphism (with ${\rm degree}(z) = 2$) such that  for $i \geq 1$ and $x \in H^s(J(C))$, we have  $\psi_n( x z^i)= \left(( \iota_{n-1,n})_* (1) \right)^i \cdot \pi_n^*(x) \in H^{s+2i}(C(n))$ if $n \geq 1$, and $0$ otherwise.  
\end{thm}

 \vskip .2cm

 \subsection{Higher Chow groups}
Given a smooth quasi-projective scheme $X$ over a field and $r \geq 0$,  Bloch \cite{Bl-86} defined the higher Chow groups $\CH^*(X, r)$. The following result generalises Collino's work to the 
higher Chow groups. 

\begin{thm} \label{thm:HCG-Structure}
 Let $I_n$ denote the following ideal of the ring $\CH^*(J(C), \cdot)[z]$. 
\[
I_n  = \begin{cases}  ((\alpha): z^{2g-1-n})& \textnormal{ if } n <  2g-1,\\
(\alpha \cdot z^{n-2g+1}) &\textnormal{ if } n \geq  2g-1.
\end{cases}
\]
For every $n\geq 0$, we then have the short exact sequence 
\begin{equation} \label{eqn:thm-for-HCG*-1}
0 \to I_n \to \CH^*(J(C), \cdot)[z] \xrightarrow{\psi_n} \CH^*(C(n), \cdot) \to  0,
\end{equation}
where $\psi_n\colon  \CH^*(J(C), \cdot)[z] \to \CH^*(C(n), \cdot)$ is a bi-graded ring homomorphism such that  for $i \geq 1$ and $x \in \CH^s(J(C), r)$, we have  $\psi_n( x z^i)= \left(( \iota_{n-1,n})_* (1) \right)^i \cdot \pi_n^*(x) \in \CH^{s+i}(C(n), r)$ if $n \geq 1$, and $0$ otherwise.  
 \end{thm}

\vskip .2cm

\subsection{Additive higher Chow groups}

The theory of  the higher Chow groups is a motivic cohomology theory that describes 
the algebraic $K$-theory.  
 Bloch and Esnault in \cite{BE1} and \cite{BE2} defined the additive higher 0-cycles.
Their hope was that these additive 0-cycle groups would serve
as a guide in developing a theory of motivic cohomology with modulus 
 which could describe 
the algebraic $K$-theory
of non-reduced schemes.  The theory of the 
additive higher Chow groups  was 
further studied by R\"ulling \cite{R}, 
Krishna-Levine \cite{KLevine} and Park \cite{Park}. In the following theorem, we obtain the structure of 
this non-homotopic invariant of the symmetric powers.

\begin{thm} \label{thm:AHCG-Structure}
 For $ r , m \geq 0$, let $M=\TCH^*(J(C), r;m)[z]$
 and let $N_n$ denote
  the $\CH^*(J(C))[z]$-submodule of $M$ defined as  
\[
N_n  = \begin{cases}  ((\alpha): z^{2g-1-n}) M &\textnormal{ if } n <  2g-1,\\
(\alpha \cdot z^{n-2g+1})M  &\textnormal{ if } n \geq  2g-1.
\end{cases}
\]
For every $n\geq 0$, we then have the following short exact sequence of $\CH^*(J(C))$-modules.
\begin{equation} \label{eqn:AHCG*-1}
0 \to N_n \to M \xrightarrow{\psi_n} \TCH^*(C(n), r;m)\to  0,
\end{equation}
where for $i \geq 1$ and $x \in \TCH^*(J(C), r;m)$, we have $\psi_n( x z^i)= \left( (\iota_{n-1,n})_* (1) \right)^i \cdot \pi_n^*(x) \in \TCH^*(C(n), r;m)$ if $n \geq 1$, and $0$ otherwise.  
\end{thm}

\vskip .2cm

\subsection{Rational $K$-theory} 
 For $X \in \ProjSm_k$, let $K(X)$ denote the algebraic $K$-theory spectra and let $K_r(X)$ denote the $r$-th stable homotopy group of $K(X)$. We prove the following result for the groups $K_r(X)_{\Q}$. Recall  that $K_r(X)_{\Q}$ is a module over the ring $K_0(X)_{\Q} \cong \CH^*(X)_{\Q}$. 
\begin{thm} \label{thm:K-thy-Structure}
 Let 
  $M=K_r(J(C))_{\Q}[z]$ and let 
   $N_n$ denote the $\CH^*(J(C))_{\Q}$-submodule of $M$ defined as 
\[
N_n  = \begin{cases}  ((\alpha): z^{2g-1-n}) M &\textnormal{ if } n <  2g-1,\\
(\alpha \cdot z^{n-2g+1})M  &\textnormal{ if } n \geq  2g-1.
\end{cases}
\]
For every $n\geq 0$, we then have the following short exact sequence of $\CH^*(J(C))_{\Q}$-modules.
\begin{equation} \label{eqn:K-thy*-1}
0 \to N_n \to K_r(J(C))_{\Q}[z]  \xrightarrow{\psi_n} K_r(C(n))_{\Q} \to  0,
\end{equation}
where for $i \geq 1$ and $x \in K_r(J(C))_{\Q}$, we have $\psi_n( x z^i)= \left( (\iota_{n-1,n})_* (1) \right)^i \cdot \pi_n^*(x) \in K_r(C(n))_{\Q}$ if $n \geq 1$, and $0$ otherwise.  
\end{thm}

 \vskip .2cm
 
 \subsection{Outline of the paper} \label{sec:Outline}
 
In  \secref{sec:Prel}, we recall the definitions of the  Chow groups and the  correspondences, 
which we use as our main tools. We also recall  the Chow motives and give sufficient conditions on a functor, defined on smooth projective schemes, such that it extends to the Chow motives of degrees $0$ and top. In \secref{sec:Symm-P-Corr}, we study the correspondences on the symmetric powers of a curve. We define particular correspondences and study their composition. Using these results about correspondences, we then generalise \cite[Theorems~1 and~2]{Co} to any functor from the category of the Chow motives of degrees $0$ and top to the category of abelian groups. 
 We prove a very general theorem in \secref{sec:thm-A} regarding the structure of various invariants of a certain sequence of closed embeddings over a fixed base in terms of the invariants of the base scheme. For the precise statement see \thmref{thm:thm-A}. We then specialise to the study of invariants of the symmetric powers of a curve in terms of the invariants of the Jacobian of the curve. In  \secref{sec:Symm-P-Corr-J}, we consider the invariants which have a ring structure. In  \secref{sec:thm-B}, we focus on the invariants which are modules over the Chow ring and prove results similar to  \secref{sec:Symm-P-Corr-J} for these invariants. In \secref{sec:App}, as applications of Theorems \ref{thm:thm-A} and \ref{thm:thm-B}, we prove our Theorems~\ref{thm:WCT-Structure}--\ref{thm:K-thy-Structure}.

\vskip .2cm

\subsection{Notations}\label{sec:Not}
 
 Given a field $k$, we let $\Sch_k$ denote the category of separated
finite type schemes over $k$ and let $\Sm_k$ denote the full subcategory of 
smooth schemes over $k$. We let $\ProjSm_k$ denote the full subcategory of $\Sm_k$ generated by smooth projective schemes over $k$. 
For a scheme $X$, we let $\Delta_X$ denote both the 
diagonal embedding $X \inj X \times X$ and its image. 

 For a smooth projective scheme $C$ over an algebraically closed field $k$ and $n\geq 0$, we let $C[n]$ denote the product of $n$-copies of $C$ and let $C(n)$ denote the $n$-th symmetric product of $C$. The quotient map is denoted by $\rho_n\colon  C[n] \to C(n)$. Note that $\rho_n$ is a finite separable morphism of degree $n!$. 
  For closed points $x_1, \dots, x_n\in C$, we denote the corresponding closed point in $C[n]$ by $[x_1, \dots, x_n]$ and denote its image in $C(n)$   by $(x_1, \dots, x_n)$.

We let ${\bf Rings}$ denote the category of associative rings with $1$ and let ${\bf CRings}$ denote  the category of commutative rings with $1$. Moreover, we let ${\bf GrRings}$ (resp. 
${\bf GrCRings}$) denote the category of graded commutative rings (resp. commutative rings that are graded). By graded commutative ring $R$, we mean a graded ring $R = \oplus_i R_i$ such that  if $x \in R_i$ and $y\in R_j$, then $xy = (-1)^{i+j} y x$. For a graded $R$-module $M$ and $n \in \Z$, we write the shifted graded module by $M[n]$, where $M[n]_i = M_{i+n}$.  For a commutative 
ring $R$, we let ${\bf GrAlg}_R$ denote the category of graded $R$-algebras.
%
 
\vskip .2cm

\section{Preliminaries } \label{sec:Prel}

In this section, we recall the definitions of the correspondences and the 
Chow motives. We prove some basic lemmas about correspondences. 
We also give a set of conditions on a functor from smooth projective  schemes to abelian groups which are sufficient to lift the functor to the Chow motives.

We begin with recalling the definition of the Chow group of algebraic cycles on schemes.

\vskip .2cm

\subsection{Chow groups} \label{sec:Prel-CH}

Let $k$ be a field and let $X \in \Sch_k$. For $i \in \Z_{\geq 0 }$, let $Z^i(X)$ denote the free abelian group generated on the set of irreducible closed subsets of $X$ of codimension $i$ and let $R^i(X) \subset Z^i(X)$ be the subgroup generated by ${\rm div}(g)$, where $g \in k(Y)^{\times}$ for an integral subvariety $Y$ of $X$ of codimension $i-1$. The Chow group of codimension $i$-cycles is defined to be the quotient $\CH^i(X):= Z^i(X)/R^i(X)$. For more details and basic properties of these groups, we refer the reader to \cite{Fulton}. 
It is worth mentioning that for a smooth scheme $X$, the direct sum $\CH^*(X) = \oplus_{i} \CH^i(X)$ forms a ring, known as the Chow ring of $X$ and for a morphism $f\colon X \to Y$, where $Y$ is smooth, we have a pull-back map $f^*\colon  \CH^*(Y) \to \CH^*(X)$.

Let $f\colon X \to Y$ be a morphism between smooth projective schemes and let $Z^i_f(Y)$ be the subgroup of $Z^i(Y)$ generated by irreducible closed subsets $W \in Z^i(Y)$ such that each irreducible component of the closed  subscheme $f^{-1}(W) \subset X$  has codimension at least  $i$. Let $R^i_f(Y) = Z^i_f(Y) \cap R^i(Y)$. 
By Fulton's construction of the pull-back map (for example, see \cite[\S~8.1, before Definition~8.1.1]{Fulton}), 
it follows that the pull-back map $f^*\colon  Z^*_f(Y) \to \CH^*(X)$ factors through $Z^*(X)$ and for $W\in Z^*_f(Y)$, the pull-back cycle $f^*(W) \in Z^*(X)$ is supported on $f^{-1}(W)$. 
We therefore have the following commutative diagram.
\begin{equation}\label{eqn:CG*-1}
\xymatrix@C.8pc{
Z^*_f(Y) \ar[r]^{f^*} \ar[d] & Z^*(X) \ar@{->>}[d]\\
\CH^*(Y) \ar[r]^{f^*} & \CH^*(X).}
\end{equation}

\vskip .2cm

 \subsection{Correspondences} \label{sec:Prel-Corr}
 
 For $X, Y \in \Sch_k$ and $i \in \Z$, we define the group of degree $i$ correspondences from $X$ to $Y$ as
 ${\rm Corr}^i(X, Y) = \oplus_{j} \CH^{d_j +i} (X_j \times Y)$, where $X = \cup_j X_j$ is the irreducible decomposition of $X$ and $d_j = \dim(X_j)$. 
 We denote the direct sum $\oplus_i \Corr^i(X,Y)$ by $\Corr(X, Y)$. 
These groups satisfy the following properties. 
 
 \begin{enumerate}
 \item 
  If $X$ is an irreducible scheme of dimension $d$, then ${\rm Corr}^i(X, Y) = \CH^{d+i}(X \times Y)$. 
  
  \item
For a correspondence $\alpha \in \Corr^i(X, Y)$ of irreducible schemes $X$ and $Y$, we have a dual correspondence $\alpha^{\op} \in \Corr^{d_X - d_Y +i}(Y, X)$ which is obtained by applying the involution $X \times Y \xrightarrow{\cong} Y \times X$ to $\alpha$. Here $d_X$ and $d_Y$ denote  the dimensions of $X$ and $Y$, respectively.

  \item Given a morphism $f\colon  Y \to X$ of irreducible schemes, the graph $\Gamma_f \subset X \times Y$ belongs to $\Corr^0(X, Y)$. We have its dual correspondence $\Gamma_f^{\op} \in \Corr^{d_X - d_Y} (Y, X)$.
     
  \item  For smooth projective  schemes $X$, $Y$ and $Z$, we let $\alpha \in \Corr^i(X, Y)$ and $\beta \in \Corr^j(Y, Z)$. Without loss of generality, we can assume that $X$ and $Y$ are irreducible. We  define the composition $\beta \circ \alpha  \in \Corr^{i+j}(X, Z)$ as follows. 
  \begin{equation} \label{eqn:comp*-0}
  \beta \circ \alpha  = p_{13*} (p_{12}^* \alpha \cdot p_{23}^* \beta),
  \end{equation}
  where $p_{12}\colon X \times Y \times Z \to X \times Y$, 
   $p_{13}$ and $p_{23}$ are respective projections.  The product, on the right hand side of \eqref{eqn:comp*-0}, is taken in the Chow ring $\CH^*(X \times Y \times Z)$ and the push-forward $p_{13*}$ exists because  $Y$ is a projective scheme. 
 \end{enumerate}

We define the category $\Co_k$ whose objects are smooth projective schemes over $k$ and the morphism set between two smooth  projective schemes $X$ and $Y$ is given by $\Mor_{\Co_k}(X, Y) =\Corr(X,Y)$. The composition of the morphisms in $\Co_k$ is determined by \eqref{eqn:comp*-0}. We have a contravariant functor $\ProjSm_k \to \Co_k$ which is 
 identity on the objects and maps a morphism $f\colon  Y \to X$ to $\Gamma_f \in \Corr(X, Y) = \Mor_{\Co_k}(X, Y)$. The following lemma yields the expression for the compositions with the graph of a morphism.

\begin{lem}$($\cite[Proposition~62.4]{EKM08}$)$ \label{lem:Corr*-1} Let $X, Y, Z, W \in \ProjSm_k$ with $\alpha \in \Corr(Z, X)$ and $ \beta \in \Corr(Y, W)$. Let $f\colon  Y \to X$ be a morphism and let $\Gamma_f \in \Corr(X, Y)$ be the correspondence associated to $f$. 
We then have
\begin{eqnarray}
\label{eqn:Corr*-2} \Gamma_f \circ \alpha &=& (\id \times f)^* \alpha \in \Corr(Z, Y), \\
\label{eqn:Corr*-3} \beta  \circ \Gamma_f  &=& (f \times \id)_* \beta \in \Corr(X, W).
\end{eqnarray}
\end{lem}

\vskip .2cm

\subsection{Chow Motives} \label{sec:Prel-CM}

Let $k$ be a field and let $\Lambda$ be an abelian group. The $\Lambda$-linear category $\sM_k(\Lambda)$ 
of the 
Chow motives over $k$, with coefficients in $\Lambda$, is defined as follows: an object of $\sM_k(\Lambda)$ is a triple $(X, p, m)$, where $X$ is a smooth projective scheme over $k$, $m \in \Z$ and $p \in \Corr^0(X, X)_{\Lambda}$ which  satisfies $p \circ p = p$. 
For the Chow motives $(X, p, m)$ and $(Y, q, n)$, the morphism set is given by 
\[
\Mor_{\sM_k(\Lambda)}((X, p, m)(Y, q, n)) = q \circ \Corr^{n-m}(X, Y)_{\Lambda} \circ p \subset \Corr^{n-m}(X, Y)_{\Lambda}.
\]
Composition law is given by the composition of the correspondences. We have a functor $h\colon  \ProjSm_k^{\op} \to \sM_k(\Lambda)$ so that $h(X) = (X, 1, 0)$ and 
$h(f) = \Gamma_f$. For a Chow motive $M=(X, p, m)$ and $i \in \Z$, we let
 $M(i)= (X, p, m +i)$. For $X \in \ProjSm_k$ and a 
  vector bundle $\pi\colon \sE \to X$  of rank $e+1$, we let 
 \[
 c_i = c_1(\pr_2^* \sO_{\P(\sE)}(1))^i \cap \Gamma_{\pi} \in \Corr^{i}(X, \P(\sE)) =  
\Mor_{\sM_k}(h(X)(-i), h(\P(\sE))).
\]
By \cite[Theorem~63.10]{EKM08}, it follows that the map
\begin{equation}\label{eqn:PBF-C-M}
\sum_{i=o}^e c_i\colon  \oplus_{i=0}^e h(X)(-i) \to h(\P(\sE))
\end{equation}
is an isomorphism in the category of the Chow motives.

We let $\sM_k(\Lambda)^{0, \dim}$ denote the full subcategory of  $\sM_k(\Lambda)$ consisting
of  the objects of the type $(X, 1, 0)$ and $(X, 1, \dim(X))$ for $X \in \ProjSm_k$. 
Note that the above functor $h\colon  \ProjSm_k^{\op} \to \sM_k(\Lambda)$ factors through a functor  $h\colon  \ProjSm_k^{\op} \to \sM_k(\Lambda)^{0, \dim}$. For a morphism $f\colon Y \to X$ in $\ProjSm_k$, we have 
\[
\Gamma_f^{\op} \in \Corr^{\dim(X)- \dim(Y)}(Y, X)_{\Lambda} = \Mor_{\sM_k(\Lambda)}((Y,1, \dim(Y)),(X, 1, \dim(X))).
\] It therefore follows that $\sM_k(\Lambda)^{0, \dim}$ is a minimal full subcategory of $\sM_k(\Lambda)$ where $\Gamma_f$, $\Gamma_f^{\op}$ and their push-forwards make sense. 

We write $\sM_k = \sM_k(\Z)$ and $\sM_k^{0, \dim} = \sM_k(\Z)^{0, \dim}$.

\vskip .2cm

\subsection{Motivic Functors} \label{sec:Mot-Func}

Let $G\colon  \ProjSm_k^{\op} \to \Ab$ be a presheaf of abelian groups on smooth projective schemes which has the following properties. 

\begin{description}

\item[P1] $G$ has push-forwards, i.e., there exists a covariant functor $G^{\prime}\colon  
\ProjSm_k \to \Ab$ such that for all $X \in \ProjSm_k$, we have $G(X) = G^{\prime}(X)$. 
For a morphism $f\colon X \to Y$, we write $f^* = G(f)$ and $f_*= G^{\prime}(f)$.

\item[P2]
The functors $G$ and $G^{\prime}$ satisfy the base change formula,  i.e., for a cartesian diagram 
\begin{equation}\label{eqn:comp*-1}
\xymatrix@C.8pc{
X^{\prime} \ar[r]^{g^{\prime} } \ar[d]^{f^{\prime}}& X \ar[d]^{f}\\
Y^{\prime} \ar[r]^{g} & Y}
\end{equation}
with $g$ being a flat morphism, we have $G(g) \circ G^{\prime} (f) = G^{\prime}(f^{\prime}) \circ G(g^{\prime})$. In other words,
\begin{equation} \label{eqn:Assum-B-C}
 g^{*} f_* =  f^{\prime}_* g^{\prime *}.
 \end{equation}

\item[P3]
For each $X \in \ProjSm_k$, the group $G(X)$ is a $\CH^*(X)$-module such that for a morphism $f \colon  Y \to X$ in $ \ProjSm_k$, the pull-back
$f^*\colon  G(X) \to G(Y)$ and the push-forward $f_*\colon  G(Y) \to G(X)$ are $\CH^*(X)$-module homomorphisms.  Note that $G(Y)$ is a $\CH^*(Y)$-module and $\CH^*(Y)$ is a $\CH^*(X)$-algebra under the ring homomorphism $\CH^*(X) \to \CH^*(Y)$. 

\item[P4] If $f\colon Y \to X$ is flat, then  $G$ satisfies the projection formula, i.e., for all $\beta \in \CH^*(Y)$ and $x \in G(X)$, we have
\begin{equation} \label{eqn:Assum-P-F}
f_*(\beta \cdot f^*(x) ) = f_*(\beta) \cdot x.
\end{equation}
\end{description}

The following lemma is well known to the experts. We include its proof  for the sake of completion. 

\begin{lem} \label{lem:Ext*-1}
Let $G\colon  \ProjSm_k^{\op} \to \Ab$ satisfy  {\bf P1-P4}. It then extends to an additive functor $F\colon  \sM_k^{0, \dim} \to \Ab$ such that $F((X,1, 0))= G(X) = F((X, 1, \dim(X)))$ for all $X \in \ProjSm_k$, and $f^* = F(\Gamma_f)$ and $f_* = F(\Gamma_f^{\op})$ for all $f\colon  Y \to X$. 
\end{lem}
 
 \begin{proof}
 Note that sending  the Chow motive $(X, p, n)$ to the smooth projective scheme $X$ and looking a morphism between the Chow motives as a correspondence, we get a well defined functor $\sM_k \to \Co$. Observe that the composite functor $\Theta\colon  \ProjSm_k^{\op} \to \Co$ is same as the functor in \secref{sec:Prel-Corr}.
 It therefore suffices to show  that $G$ extends to $F\colon  \Co \to \Ab$ such that $f^* = F(\Gamma_f)$ and $f_* = F(\Gamma_f^{\op})$ for all $f\colon  Y \to X$.
 
 We define $F\colon  \Co \to \Ab$ so that on objects, we have $F(X) = G(X)$ and for a correspondence $\alpha \in \Corr(X, Y) = \CH^*(X \times Y)$, the group homomorphism $F(\alpha)\colon  G(X) \to G(Y)$ is given by 
 \[
 F(\alpha)(x) = p_{Y *} (\alpha \cdot p_X^* (x)),
 \] 
 where $p_Y\colon  X \times Y \to Y$ and $p_X \colon  X \times Y \to X$ are the projection maps. 
 First, we verify that for a morphism $f\colon  Y \to X$, we have $f^* = F(\Gamma_f)$ and $f_* = F(\Gamma_f^{\op})$. This follows because for  $x \in G(X)$ and $y \in G(Y)$, we have 
 \begin{eqnarray*}
 f^*(x) & = &  p_{Y*} \circ \iota_{\Gamma_f *} (f^*(x))\\
 & = & p_{Y*} \left(  \iota_{\Gamma_f *} (\iota_{\Gamma_f}^* \circ p_X^* (x)) \right)\\
 & = & p_{Y*} \left( \iota_{\Gamma_f *}(1)  \cdot  p_X^* (x) \right) \\
 & = &  p_{Y*} \left( \Gamma_f   \cdot  p_X^* (x) \right)\\
 & = & F(\Gamma_f)(x) 
 \end{eqnarray*}
 and if $\sigma \colon  X \times Y \to Y \times X$ denotes the involution, then 
 \begin{eqnarray*}
 f_*(y) & = &  p_{X*} \circ \sigma_* \circ \iota_{\Gamma_f *} (f_*(y))\\
 & = & p_{X*} \left( \sigma_* \circ  \iota_{\Gamma_f *} (\sigma^* \circ \iota_{\Gamma_f}^* \circ p_Y^* (y)) \right)\\
 & = & p_{X*} \left(\left( \sigma_* \circ \iota_{\Gamma_f *}(1)\right)  \cdot  p_Y^* (y) \right) \\
 & = &  p_{X*} \left( \Gamma_f^{\op}   \cdot  p_Y^* (y) \right)\\
 & = & F(\Gamma_f^{\op})(y). 
 \end{eqnarray*}
 We now verify that $F(\id) = \id$ and $F(\beta \circ \alpha) = F(\beta) \circ F(\alpha)$ for all $\alpha \in \Corr(X, Y)$ and $\beta \in \Corr(Y, Z)$. 
Note that $F(\id_{\Theta(X)})= \id_{X}^* = \id_{G(X)}$. For $\alpha  \in \Corr(X, Y)$, $\beta \in \Corr(Y, Z)$ 
and $x \in G(X)$, we consider the  diagram 
 \begin{equation}\label{eqn:Exn*-1-1}
\xymatrix@C.8pc{
&  & X \times Y \times Z    \ar[ld]_-{p_{12}} \ar[d]^-{p_{23}}  \ar[rd]^-{p_{13}}  &  \\
 & X \times Y \ar[dl]_-{p_X} \ar[d]^-{p_Y}& Y \times Z \ar[ld]^-{p_Y^{\prime}} \ar[rd]^-{p_Z^{\prime}}& X \times Z \ar[d]^-{p_Z} \ar[dr]^-{p_X^{\prime}}\\
 X & Y & & Z & X .}
\end{equation}
Observe that both squares in \eqref{eqn:Exn*-1-1}
commute.  With these notations, we have 
  \begin{eqnarray*}
 F(\beta \circ \alpha) (x) & = & 
 F\left(p_{13*}\left(p_{23}^* (\beta) p_{12}^*(\alpha) \right)  \right) (x)\\ 
 & =& p_{Z *} \left( \left(p_{13*}\left(p_{23}^* (\beta)  p_{12}^*(\alpha) \right)  \right) \cdot p_X^{\prime *} (x) \right)\\
 & =^1& p_{Z *} \circ p_{13*} \left ( \left(p_{23}^* (\beta)  p_{12}^*(\alpha) \right)  \cdot \left ( p_{13}^* \circ p_X^{\prime *} (x) \right) \right) \\
  & =^2& p_{Z * }^{\prime} \circ p_{23*} \left ( \left(p_{23}^* (\beta)  p_{12}^*(\alpha) \right)  \cdot \left ( p_{13}^* \circ p_X^{\prime *} (x) \right) \right)\\
  & =^3& p_{Z * }^{\prime} \circ p_{23*} \left ( \left(p_{23}^* (\beta)  p_{12}^*(\alpha) \right)  \cdot \left ( p_{12}^* \circ p_X^{*} (x) \right) \right)\\
   & =& p_{Z * }^{\prime} \circ p_{23*} \left(p_{23}^* (\beta) \cdot \left( p_{12}^*(\alpha)   \cdot \left ( p_{12}^* \circ p_X^{*} (x) \right) \right) \right)\\
  & =& p_{Z * }^{\prime} \circ p_{23*}   \left(p_{23}^* (\beta) \cdot (p_{12}^* (\alpha \cdot p_X^*(x))) \right)\\
  & =^4& p_{Z * }^{\prime}  \left( \beta \cdot (p_{23*} \circ p_{12}^* (\alpha \cdot p_X^*(x)))  \right)\\
  & =^5& p_{Z * }^{\prime}  \left( \beta \cdot \left( p_Y^{\prime *} \circ p_{Y *} (\alpha \cdot p_X^*(x)) \right)  \right) \\
  & = & p_{Z * }^{\prime}  \left( \beta \cdot \left( p_Y^{\prime *} (F(\alpha) (x)) \right)\right)\\
  & = & F(\beta) \circ F(\alpha) (x),
 \end{eqnarray*}
 where the equality $=^1$ follows from \eqref{eqn:Assum-P-F}, the equalities $=^2$ and $=^3$ follow from the commutative squares in \eqref{eqn:Exn*-1-1}, the equality $=^4$ follows because $p_{23*}$ is a homomorphism of $\CH^*(Y \times Z)$-modules, and the equality $=^5$ follows from  \eqref{eqn:Assum-B-C}. By construction, $F$ is additive. 
 This completes the proof. 
  \end{proof}

\begin{remk} 

Observe that the functor $G$ in \lemref{lem:Ext*-1} actually extends to a functor 
$F \colon \sM_k \to \Ab$ such that the abelian groups $F(X, p, m)$ depend only on 
$X$. 
By introducing the smaller category $ \sM_k^{0, \dim}$ and writing 
 \lemref{lem:Ext*-1} in its current form, we want to emphasise that 
 $ \sM_k^{0, \dim}$ is a minimal  subcategory of $\sM_k$ that 
 is sufficient for 
 the results of next sections. 
\end{remk}

 \begin{remk} \label{rem:Gr-Fun}
We can also consider the class of functors $F\colon  \sM_k \to {\bf GrAb}$ such that 
 $F((X, 1, m)) \cong F((X, 1, 0))[m]$. Compositing them with the forgetful functor $ {\bf GrAb} \to  {\bf Ab}$
 and restricting to $ \sM_k^{0, \dim}$, we obtain the functors as in \lemref{lem:Ext*-1}.
 \end{remk}

\vskip .2cm

\section{Symmetric powers and Correspondences}\label{sec:Symm-P-Corr}

 In this section, we recall some basic properties of the symmetric powers of a curve and define some special correspondences on these smooth schemes. These correspondences play the most  crucial role in the study of ``nice'' invariants of the symmetric powers. Towards the end of this section, we use these correspondences to generalise  \cite[Theorems~1 and~2]{Co}.

 \vskip .2cm

\subsection{Correspondences on the symmetric powers}\label{sec:Symm-P-Corr-On-Symm}

Let $C$ be a  smooth projective curve over an algebraically closed field $k$ and let $p \in C$ be a closed point. We fix integers $n\geq m \geq 0$.
The embedding $\iota_p\colon  \Spec k \inj C$ then yields a closed embedding $j_{m, n}\colon  C[m] \to C[n]$ such that on the closed points, we have
$j_{m,n} ([x_1, \dots, x_m]) = [x_1, \dots, x_m, p, \dots, p]$. Since this embedding respects the action of the 
symmetric groups $S_i$ on $C[i]$ (for $i=m, n$), it induces an embedding $\iota_{m,n}\colon C(m) \to C(n)$. We also have the projection map $\pr_{n, m}\colon  C[n] \to C[m]$ onto the first $m$-components. Observe that this map does not respect the action of  $S_i$ and therefore does not induce a morphism from $C(n) \to C(m)$. But we can consider a correspondence from $C(m)$ to $C(n)$ (which plays the role of the projection morphism $\pr_{n,m}$)
as follows. Let 
\begin{equation} \label{eqn:Proj-corr}
\Gamma = (\rho_m \times \rho_n)_* (\Gamma_{\pr_{n, m}}) \in \Corr^0 (C(m), C(n)),
\end{equation}
where $ \Gamma_{\pr_{n, m}} \in \Corr^0(C[m], C[n])$.

Observe that $\pr_{n,m} \circ j_{m,n} = \id \colon  C[m] \to C[m]$. We therefore have that $\Gamma_{j_{m,n}} \circ \Gamma_{\pr_{n, m}} = \Delta_{C[m]} \in \Corr^0(C[m], C[m])$.  We prove that the element $\Gamma_{\iota_{m, n}}  \circ \Gamma$ is very close to the element $\Delta = \Delta_{C(m)}$. This result is similar to \cite[Corollary~1]{Co} and 
essentially follows from \cite[Proposition~1]{Co}.  To see the proof with complete details, we break it into a couple of Lemmas.

\begin{lem}\label{lem:Col-lem*-1.5}
For $m \geq i$, let $\Gamma_{m,i} = (\rho_{i} \times \rho_{m}) (\Gamma_{\pr_{m, i }}) \subset C(i) \times C(m) \subset C(m) \times C(m)$. Then for $n\geq m$, we have
\begin{equation} \label{eqn:inv-image*-1}
(\id \times \iota_{m, n})^{-1} \Gamma_{n, m} = \Delta \cup \left( \cup_{i< m} \Gamma_{m, i} \right) \subset C(m) \times C(m).
\end{equation}
\end{lem}
\begin{proof}
We look at both sides of \eqref{eqn:inv-image*-1} as closed subschemes of $C(m) \times C(n)$.  On $k$-rational points, we have
\begin{eqnarray*}
&& (\id \times \iota_{m, n})^{-1} \Gamma_{n,m} \\
&& \hskip 1cm  = \{(y_1, \dots, y_m), (x_1, \dots, x_m, p, \dots, p)~|~ \{y_1, \dots, y_m\} \subset \{x_1, \dots, x_m, p, \dots, p\} \}\\
 & & \hskip 1cm  =  \cup_{i \leq m} \{(y_1, \dots, y_i, p, \dots, p), (x_1, \dots, x_m, p, \dots, p)~|~ \{y_1, \dots, y_i\} \subset \{x_1, \dots, x_m\} \}\\
 & & \hskip 1cm =^1 \Delta \cup  \left( \cup_{i< m} \Gamma_{m, i} \right) \subset C(m) \times C(n), 
\end{eqnarray*}
where the equality $=^1$ follows because 
$\Gamma_{m, i} =  \{(y_1, \dots, y_i), (x_1, \dots, x_m)~|~ \{y_1, \dots, y_i\} \subset \{x_1, \dots, x_m\} \} \subset C(i) \times C(m)\subset C(m) \times C(n)$. 
\end{proof}

Observe that the embedding $\iota_{m-1, m}\colon C(m-1) \inj C(m)$ induces a homomorphism $(\iota_{m-1, m}\times \id)_*\colon  \Corr^{i}(C(m-1), C(m)) \to \Corr^{i}(C(m), C(m))$. We now prove 
one of the most crucial lemmas.

\begin{lem}\label{lem:Col-lem*-2}
There exists $Y \in \Corr^{0}(C(m-1), C(m))$ such that
 \begin{equation} \label{eqn:Col-lem*-2-0}
\Gamma_{\iota_{m, n}}  \circ \Gamma - \Delta=(\id \times \iota_{m, n})^{*} \Gamma - \Delta = (\iota_{m-1,m}\times \id)_* (Y) = Y \circ \Gamma_{\iota_{m-1, m}} \in \Corr^0(C(m), C(m)).
\end{equation}
\end{lem}
\begin{proof}
The first equality and the last equality in \eqref{eqn:Col-lem*-2-0} follow from \lemref{lem:Corr*-1}. To prove      the middle equality, we let $C_0(m) = C(m) \setminus C(m-1)$ and let $\phi \colon  C_0(m) \times C_0(m) \to C(n) \times C(m)$ be the composition of the open embedding $C_0(m) \times C_0(m) \inj C(m) \times C(m)$ with the closed embedding $\id \times \iota_{m,n} \colon  C(m) \times C(m) \inj C(m) \times C(n) $. 
 Consider the diagram
\begin{equation}\label{eqn:pull-back*-1}
\xymatrix@C.6pc{
 & & Z^{m}_{\id \times \iota_{m,n}}(C(m)\times C(n)) \ar[d]^{(\id \times \iota_{m,n})^*}  \ar[rd]^{\phi^*} & &\\ 
0 \ar[r] & Z^{m-1}(X) \ar[r] & Z^m(C(m) \times C(m)) \ar[r]& Z^m(C_0(m) \times C_0(m)) \ar[r]& 0,}
\end{equation}
where $X = C(m-1) \times C(m) \cup C(m) \times C(m-1)\subset C(m) \times C(m)$ is the compliment of $C_0(m) \times C_0(m)$. In the notations of \secref{sec:Prel-CH}, 
it follows from \lemref{lem:Col-lem*-1.5} that $ \Gamma_{n,m} \in Z^{m}_{\id \times \iota_{m,n}}(C(m)\times C(n))$. Note that $\Gamma$ is the class of $\Gamma_{n,m}$ in $\Corr^0(C(m), C(n))$. 
By \secref{sec:Prel-CH}, the cycle 
 $(\id \times \iota_{m,n})^* (\Gamma)$ is then  supported on 
$(\id \times \iota_{m,n})^{-1} (\Gamma_{n,m}) = \Delta \cup Z$, where $Z \subset C(m-1) \times C(m) \subset C(m) \times C(m)$. Moreover, 
\cite[Proposition~1]{Co} implies that $\phi^* (\Gamma_{n,m}) = \Delta_{C_0(m)}$. Since the bottom row in \eqref{eqn:pull-back*-1} is exact, it follows that $(\id \times \iota_{m, n} )^* \Gamma - \Delta = (\iota_{m-1,m} \times \id )_* (Y)$ for some $Y \in \CH^{m-1}(C(m-1)\times C(m)) =  \Corr^{0}(C(m-1) , C(m))$. This completes the proof.
\end{proof}

\vskip .2cm

\subsection{Functors on the symmetric powers of a curve} \label{subsec:F(C(n))}
%
In this section, we investigate the relation between the groups $F(C(n))$, where $F\colon \sM_k(\Lambda)^{0, \dim} \to {\Ab}$. We first prove a generalisation of \cite[Theorem~1]{Co}.

\begin{lem} \label{lem:inj}
Let $F\colon \ProjSm_k^{\op} \to \Ab$ be a functor that extends to an additive functor 
$F\colon  \sM_k(\Lambda)^{0, \dim} \to \Ab$. Then for all $0 \leq m \leq n$, 
the natural map 
\[
(\iota_{m,n})_* := F(\Gamma_{\iota_{m,n}}^{\op})\colon  F((C(m), 1, m)) \to  F((C(n), 1, n))
\]
 is injective. 
\end{lem}
\begin{proof}
We prove this by induction on $m\geq 0$. For the base case, observe that $C(0) = \Spec k$ and 
the composition $C(0) \xrightarrow{\iota_{0, n}} C(n) \xrightarrow{\chi_n} \Spec k$ is the identity map, where $\chi_n\colon  C(n) \to \Spec k$ is the structure morphism.  It therefore follows that 
\[
\id = F(\Gamma_{\chi_n}^{\op}) \circ F(\Gamma_{\iota_{0,n}}^{\op})\colon  F((C(0), 1, 0)) \to F((C(0), 1, 0)).
\]
In particular, $F(\Gamma_{\iota_{0,n}}^{\op})$ is injective, which proves the base case. We now assume  $m \geq 1$. By the induction hypothesis, for all $n\geq m-1$, the map $(\iota_{m-1,n})_*\colon  F((C(m-1)), 1, m-1) \to F((C(n), 1, n))$ is injective. For $n \geq m$, we consider the diagram 
\begin{equation}\label{eqn:inj*-1}
\xymatrix@C3pc{
 F((C(m-1), 1, m-1)) \ar@{^{(}->}[r]^-{(\iota_{m-1, m})_*} \ar@{^{(}->}[rd]^-{(\iota_{m-1, n})_*} & F((C(m), 1, m)) \ar[d]^-{(\iota_{m,n})_*}\\
  & F((C(n), 1, n)).}
\end{equation}
Note that for all $j \geq i$, $(\iota_{i, j})_* = F( \Gamma_{i, j}^{\op})$ and hence the diagram  \eqref{eqn:inj*-1} is commutative. Let $\Gamma \in \Corr^0(C(m), C(n))$ be as in \eqref{eqn:Proj-corr}. Then $\Gamma^{\op} \in \Mor_{\sM_k(\Lambda)^{0, \dim}}((C(n), 1, n), (C(m), 1, m))$. 

Moreover, by \lemref{lem:Col-lem*-2}, there exists $Y \in \Corr^{0}(C(m-1), C(m))$ such that 
 the following diagram, in the category $\sM_k(\Lambda)^{0, \dim}$, commutes. 
\begin{equation}\label{eqn:inj*-2}
\xymatrix@C3pc{
(C(m), 1, m) \ar[r]^{Y^{\op}} \ar[dr]_-{\Gamma^{\op} \circ \Gamma_{\iota_{m,n}}^{\op} - \Delta^{\op}  } & 
  (C(m-1), 1, m-1)\ar[d]^{ \Gamma_{\iota_{m-1, m}}^{\op}}\\
 & (C(m), 1,m). } 
\end{equation}
Observe that since the category $\sM_k(\Lambda)^{0, \dim}$ is $\Lambda$-linear, the morphism sets in this category have an abelian group structure and hence the diagonal arrow in \eqref{eqn:inj*-2} is well defined. 
Since $F$ is an additive functor and $\Delta^{\op} = \id$,  we have 
\[
F(\Gamma^{\op}) \circ F(\Gamma_{\iota_{m,n}}^{\op}) - \id=  F(\Gamma_{\iota_{m-1, m}}^{\op})\circ F(Y^{\op}).
\]
Let $x \in F((C(m), 1, m))$ such that $F(\Gamma_{\iota_{m,n}}^{\op})(x) = (\iota_{m,n})_*(x) = 0$.  We then have 
\[
x +  F(\Gamma_{\iota_{m-1, m}}^{\op})\circ F(Y^{\op})(x ) =  F(\Gamma^{\op}) \circ F(\Gamma_{\iota_{m,n}}^{\op})(x) = 0. 
\]
In particular, $x =  F(\Gamma_{\iota_{m-1, m}}^{\op}) (y)$, where $y\in F((C(m-1),1,m-1)) $. 
But then 
\[
0 = F(\Gamma_{\iota_{m,n}}^{\op})(x) = F(\Gamma_{\iota_{m,n}}^{\op}) \circ F(\Gamma_{\iota_{m-1, m}}^{\op}) (y) = F(\Gamma_{\iota_{m-1, n}}^{\op}) (y).
\]
By the induction hypothesis, we have $ y = 0$ and hence $x = F(\Gamma_{\iota_{m-1, m}}^{\op}) (y) = 0$. This completes the induction and hence  the proof. 
\end{proof}

We now prove a generalisation of \cite[Theorem~2]{Co}. 

\begin{lem} \label{lem:surj}

Let $F\colon \ProjSm_k^{\op} \to \Ab$ be a functor that  extends to an additive functor 
$F\colon  \sM_k(\Lambda)^{0, \dim} \to \Ab$. Then for all $0 \leq m \leq n$, 
the pull-back map 
\[
(\iota_{m,n})^*\colon  F(C(n)) \to  F(C(m))
\]
 is surjective. 
\end{lem}
 
\begin{proof}
We first note that the pull-back map $(\iota_{m,n})^*\colon  F(C(n)) \to  F(C(m))$ agrees with 
the map $F(\Gamma_{\iota_{m,n}}) \colon  F((C(n), 1, 0)) \to F((C(m), 1, 0))$. We switch between these two notations without mentioning it explicitly. 
 We now prove the result by induction on $m\geq 0$. For $m =0$, observe as before  that 
$\id =F(\Gamma_{\iota_{0,n}}) \circ F(\Gamma_{\chi_n}) \colon  F((C(0), 1, 0)) \to F((C(0), 1, 0))$ and therefore $(\iota_{0,n})^*\colon  F(C(n)) \to F(C(0))$ is surjective. This proves the base case of the induction. We now assume  $m\geq 1$. By the induction hypothesis,   the map $(\iota_{m-1,n})^*\colon  F(C(n)) \to  F(C(m-1))$ is surjective for all $n \geq m-1$. For $n \geq m$, we consider the commutative diagram
\begin{equation}\label{eqn:surj*-1}
\xymatrix@C3pc{
  F(C(n)) \ar@{->>}[rd]^-{(\iota_{m-1,n})^*} \ar[d]_-{(\iota_{m,n})^*}& \\
F(C(m)) \ar@{->>}[r]_-{(\iota_{m-1,m})^*} & F(C(m-1)).}
\end{equation}
Let $x \in F(C(m))$. By the induction hypothesis, there exists $y \in F(C(n))$ such that 
$(\iota_{m-1,n})^*(y) = (\iota_{m-1,m})^* x$. In particular, $(\iota_{m-1,m})^* z = 0$, where $z= x - (\iota_{m,n})^*(y)$.
Since $F$ is an additive functor, \lemref{lem:Col-lem*-2} yields that 
\[
F(\Gamma_{\iota_{m,n}}) \circ F(\Gamma) - \id = F(Y) \circ F(\Gamma_{\iota_{m-1,m}}), 
\]
where $Y \in \Corr^0(C(m-1), C(m)) = \Mor_{\sM_k(\Lambda)^{0, \dim}}( (C(m-1),1, 0), (C(m), 1, 0))$. Since $F(\Gamma_{\iota_{m-1,m}})(z) =  (\iota_{m-1,m})^* z = 0$, it follows that 
\[
z = F(\Gamma_{\iota_{m,n}}) \circ F(\Gamma) (z) = (\iota_{m,n})^* (F(\Gamma) (z)). 
\]
 We therefore have 
 \[
 x = z + (\iota_{m,n})^*(y) = (\iota_{m,n})^* (y+ F(\Gamma) (z)). 
 \]
 This completes the induction and hence the proof. 
 \end{proof}
 
 \vskip .2cm

\section{General Theorems}

In this section, we prove a general theorem regarding the structure of various ring-valued invariants of a certain sequence of closed embeddings over a fixed base in terms of the invariants of the base scheme. We then specialise to the study of ring-valued invariants of the symmetric powers of a curve in terms of the invariants of the Jacobian of the curve. 

\subsection{Theorem~A}\label{sec:thm-A}
Let $k$ be a field. We consider the functors $G\colon  \ProjSm_k^{\op} \to {\bf Rings}$ satisfying the following properties. 

\begin{description}

\item[PBF] \label{Property:PBF} Let $X \in \ProjSm_k$ and let $\sE \to X$ be a vector bundle of rank $r+1$. Let $\psi\colon  \P(\sE) \to X$ denote the projective bundle associated to $\sE$. The functor 
$G$ satisfies the  projective bundle formula, i.e., there is a natural isomorphism of $G(X)$-algebras 
\begin{equation} \label{eqn:P-B-F}
 \frac{G(X)[z]}{\<z^{r+1} + c_1(\sE, G) z^r + \cdots + c_{r+1}(\sE, G)\>} \xrightarrow{\cong} G(\P(\sE)).
\end{equation}
 We call the element $c_i(\sE, G)$ the $i$-th Chern class of $\sE$ in $G(X)$ and denote the image of $z$ by $\zeta(\sE, G)$. By the naturality of the above isomorphism, we mean that given a morphism $f\colon \sE_1 \to \sE_2$ of vector bundles over $X$, we have $G(f) (\zeta(\sE_2, G) ) = \zeta(\sE_1, G)$.

\item[Extn]  \label{Property:Extn} The functor $G$
  extends to an additive functor $F\colon  \sM_k(\Lambda)^{0, \dim} \to \Ab$. We therefore 
   have the following commutative diagram. 
 \begin{equation}\label{eqn:general*-1}
\xymatrix@C3pc{
\ProjSm_k^{\op}  \ar[r]^G \ar[d]^{h} & {\bf Rings} \ar[d]\\
\sM_k(\Lambda)^{0, \dim} \ar[r]^{F} & \Ab,}
\end{equation}
 where the right vertical arrow is the forgetful functor. Assume further that for $X \in \ProjSm_k$,  we have $F((X, 1, 0)) =G(X) =  F((X, 1,\dim (X)))$. We denote this ring by $F(X)$. We use notations $F(X)$ and $G(X)$ interchangeably.  
 For a morphism $f\colon X \to Y$, we denote by $f^*\colon  G(Y) \to G(X)$ the ring homomorphism $G(f)\colon  G(Y) \to G(X)$, or equivalently the ring homomorphism  $F(\Gamma_f) \colon  F((Y, 1, 0)) \to F((X, 1, 0))$ and we denote by $f_*\colon  G(X) \to G(Y)$ the group homomorphism $F(\Gamma_f^{\op})\colon  F((X, 1, \dim (X))) \to F((Y, 1, \dim (Y)))$.

\item[PF] \label{Property:PF} The functor $F\colon  \sM_k(\Lambda)^{0, \dim} \to \Ab$ satisfies the projection formula, i.e., 
for all $f\colon X \to Y$, $y \in F(Y)$ and $x\in F(X)$, we have 
\begin{equation} \label{eqn:P-F}
f_*(f^*(y) x) = y f_*(x) \textnormal{ and } f_*(x f^*(y)) = f_*(x) y.
\end{equation}
\end{description}

We now prove our most general result regarding the relations between the structure of the ring theoretical functor $G(-)$ applied to various schemes. 
 
\begin{thmA} \label{thm:thm-A} 
Let $G\colon \ProjSm_k^{\op} \to {\bf Rings}$ be a functor satisfying properties  {\bf PBF} \eqref{eqn:P-B-F},  {\bf Extn} \eqref{eqn:general*-1} and {\bf PF} \eqref{eqn:P-F}.
Let $X \in \ProjSm_k$ and $\dim(X)=d$. Let $Y_0 \subset Y_1 \subset \cdots \subset Y_n \subset \cdots$ be a sequence of objects  in $(\ProjSm_k)_{/X}$ such that 
\begin{enumerate}
\item  $Y_0 = \Spec k$, 
\item for all $i \geq 0$, $Y_i $ is an irreducible closed subscheme of $Y_{i+1}$ and $\dim (Y_i) = i$,
\item for each $n \geq m \geq 0$, the pull-back map $G(Y_n) \xrightarrow{\iota_{m,n}^*} G(Y_m)$ is surjective and the push-forward map  $G(Y_m) \xrightarrow{(\iota_{m,n})_*} G(Y_n)$ is injective,
\item there exists $l\geq 1$ such that for all $n \geq l$,  $Y_{n}= \P(\sE_n) \to X$ is a projective bundle over $X$ with  $(\zeta(\sE_n, G))^i = (\iota_{n-i,n})_* (1) \in G(Y_n)$ for all $i\leq n$,  and $(\zeta(\sE_n, G))^i = 0$ for all $i \geq n+1$. 
\item  for all $i \in \Z$, there exists $u_i \in G(X)$ such that $u_i = c_i (\sE_n, G) $ for all $n\geq l$ and $0 \leq i \leq n- d + 1$, and $u_i = 0$ for all $i < 0$ and $i > d$.  
\end{enumerate}
Let $\alpha = \sum_{0 \leq i \leq d} u_i z^{d-i} \in G(X)[z]$ and let $I_n \subset G(X)[z]$ be an ideal defined as
\[
I_n  = \begin{cases}  ((\alpha)\colon  z^{l-n}) & \textnormal{ if } n <  l ,\\
(\alpha \cdot z^{n-l}) & \textnormal{ if } n \geq  l.
\end{cases}
\]
For every $n\geq 0$, we then have the following short exact sequence of $G(X)$-modules.
\begin{equation} \label{eqn:thm-A*-1}
0 \to I_n \to G(X)[z] \xrightarrow{\psi_n} G(Y_n) \to 0, 
\end{equation}
where $\psi_n$ maps $z$  to $(\iota_{n-1,n})_* (1) \in G(Y_n)$ if $n \geq 1$, and $0$ otherwise.  
\end{thmA}

\begin{proof}
Let $\zeta_j = (\iota_{j-1, j})_* (1) \in G(C(j))$ for all $j \geq 1$, and $\zeta_0 = 0$. 
We first prove that for all $n\geq l$ and $m \leq n$, we have 
\begin{equation}\label{eqn:thm-A*-2}
\iota_{m, n}^* (\zeta(\sE_n,G)) = \zeta_m.
\end{equation}
 By the assumption (3),  the map $(\iota_{m, n})_*$ is injective and therefore  it suffices to show that 
$(\iota_{m, n})_* \circ \iota_{m, n}^* (\zeta(\sE_n, G))  = (\iota_{m, n})_* ( \zeta_m)$.
This follows from the projection formula and by the assumption (4). 
More precisely, for $n \geq m \geq 1$, we have 
\begin{eqnarray*}
(\iota_{m, n})_* \circ \iota_{m, n}^* (\zeta(\sE_n, G)) & =^1 & \zeta(\sE_n, G) \cdot \left( (\iota_{m, n})_* (1)\right) \\
&=^2& \zeta(\sE_n, G) \cdot  (\zeta(\sE_n,G))^{n-m} \\
& =& (\zeta(\sE_n, G))^{n-m+1}\\
& =^3 &  (\iota_{m-1, n})_* (1)\\
& =&  (\iota_{m, n})_*\circ  (\iota_{m-1, m})_* (1)\\
& =^4 &  (\iota_{m, n})_* ( \zeta_m),
\end{eqnarray*}
where the equality $=^1$ follows from the projection formula \eqref{eqn:P-F}, the equalities $=^2$ and $=^3$ follow from the assumption (4) because $n \geq l$. Finally, the equality $=^4$ follows from the definition of $\zeta_m$. Moreover, if $m=0$, then 
$(\iota_{0, n})_* \circ \iota_{0, n}^* (\zeta(\sE_n,G)) =  \zeta(\sE_n, G) \cdot \left( (\iota_{0, n})_* (1)\right) = \zeta(\sE_n, G) \cdot  (\zeta(\sE_n, G))^{n} = (\zeta(\sE_n, G))^{n+1} = 0$. This proves the equality \eqref{eqn:thm-A*-2}.

If $n\geq l$, then the theorem follows from the projective bundle formula for $\sE_n$. Indeed, if $n \geq l$, then ${\rm rank}(\sE_n) = n-d+1$ and by the assumption (5), we have $c_i(\sE_n, G) = u_i$ for all $0 \leq i \leq n- d + 1$. The short exact sequence  \eqref{eqn:thm-A*-1} now follows from \eqref{eqn:P-B-F} and \eqref{eqn:thm-A*-2} for the case $m=n$. If $1\leq n < l$, then we define a ring homomorphism $\psi_n\colon  G(X)[z] \to G(Y_n)$ such that the following diagram commutes. 
\begin{equation}\label{eqn:thm-A*-4}
\xymatrix@C3pc{
G(X)[z] \ar@{->>}[r]^{\psi_l} \ar[rd]_{\psi_n}&G(Y_l) \ar@{->>}[d]^{\iota_{n, l}^*} \\
& G(Y_n).}
\end{equation}
Observe that $\psi_l$ is surjective by the projective bundle formula for $\sE_l$ and the right 
vertical arrow ${\iota_{n, l}^*}$ is surjective by the assumption  (3). In particular, $\psi_n$ is surjective. Moreover, by \eqref{eqn:thm-A*-2}, we have 
 $\psi_n(z) = {\iota_{n, l}^*} (\psi_l (z)) = {\iota_{n, l}^*} (\zeta(\sE_l, G)) = \zeta_n = (\iota_{n-1, n})_* (1) \in G(Y_n)$. It now suffices to show that $I_n = \ker(\psi_n)$, which follows from the following computations. 
 \begin{eqnarray*}
  0 = \psi_n(x) &\iff& 0 =  {\iota_{n, l}^*} \circ \psi_l (x)\\
  & \iff^1 &  0 =  (\iota_{n, l})_* (\iota_{n, l}^* ( \psi_l (x))\cdot 1)\\
    & \iff^2 & 0 = \psi_l(x) \cdot  (\iota_{n, l})_*(1)\\
    & \iff^3 & 0 = \psi_l(x) \cdot (\zeta(\sE_l, G)^{l-n})\\
    & \iff^4 & 0 = \psi_l(x) \cdot  \psi_l(z)^{l-n}\\
    & \iff & 0 = \psi_l(x \cdot z^{l-n})\\
    & \iff & x \cdot z^{l-n} \in \ker(\psi_l) = I_l = (\alpha)\\
    & \iff & x \in ((\alpha): z^{l-n}),
 \end{eqnarray*}
 where $\iff^1$ follows by the assumption  $3$, $\iff^2$ follows from the projection formula, 
 $\iff^3$ follows from the assumption  $4$ and $\iff^4$ follows because by the assumption  $4$, we have $\zeta(\sE_l,G) = (\iota_{n-1, n})_*(1)= \psi_l(z)$. This completes the proof for $n\geq 1$. Note that $n=0$ case can be proven on the similar lines with the assumption $Y_{-1} = \emptyset$ and $(\iota_{-1, 0})_* = 0$. 
\end{proof}

 \vskip .2cm

 \subsection{Application to the symmetric powers of a curve} \label{sec:Symm-P-Corr-J}
 
We follow the notations from \secref{sec:Symm-P-Corr-On-Symm} and let $J(C)$ denote the Jacobian of the curve $C$. Then $J(C)$ is an abelian variety of dimension $g$, where $g$ is the genus of $C$.
 Recall that the closed points of $J(C)$ are  given by  degree-zero divisor classes on $C$, i.e, degree-zero divisors modulo rational equivalences.

 For a closed point  $p \in C$ and $n\geq 0$, we have 
 the morphism $\pi_n\colon C(n) \to J(C)$ such that
 $\pi_n((x_1, \dots, x_n)) = \sum_i (x_i - p) \in J(C)$.
 For a closed point $y\in J(C)$, the fiber $\pi_n^{-1}(y)$ is the complete linear system of effective divisors equivalent to the divisor class $y + n p$. It then follows that the fibers of $\pi_n$ are projective spaces. Moreover, if $n \geq 2g-1$, then Mattuck \cite{Mat} showed that $C(n)$ is a projective bundle over $J(C)$, i.e., for each $n \geq 2g-1$, there exists a vector bundle $\sE_n$ such that $\P(\sE_n) \cong C(n)$ as schemes over $J(C)$ and 
\begin{equation}\label{eqn:Con-Chern-C}
c_1( \sO(1)_{C(n)}) = (\iota_{n-1, n})_* (1) \in \CH^*(C(n)).
\end{equation} 
 He \cite{Mat} also computed Chern classes of $\sE_n$ and showed that 
 for $n \geq 2g-1$ and $ 0 \leq i \leq n-g+1$, we have 
 \begin{equation} \label{eqn:Chern-C*-1}
 c_i(\sE_n) = (-1)^i \theta_* \circ \pi_{g-i *} (C(g-i)) \in \CH^{i}(J(C)),
 \end{equation}
 where $\theta\colon  J(C) \to J(C)$ is the automorphism such that $\theta(y ) = (-y + c)$ with $c\in J(C)$ being the class of the canonical divisor on $C$.

We now specialise \thmref{thm:thm-A} to the sequence $C(n)$ but we still take a general functor which satisfies the projection formula, the projective bundle formula and certain assumptions similar to that  in \thmref{thm:thm-A}. Observe  that \thmref{thm:thm-A} along with Lemmas \ref{lem:inj} and \ref{lem:surj}  immediately implies the following.

 \begin{thm} \label{cor:thm-A}
 Let $G\colon  \ProjSm_k^{\op} \to {\bf Rings}$ be a presheaf on smooth projective schemes which 
 satisfies {\bf PBF}   \eqref{eqn:P-B-F}, {\bf Extn} \eqref{eqn:general*-1} and {\bf PF} \eqref{eqn:P-F}.  
 With notations as above, and as in  {\bf PBF} and {\bf Extn}, assume that 
  \begin{enumerate}
  \item  for all $n\geq 2g-1$, $(\zeta(\sE_n, G))^i = (\iota_{n-i,n})_* (1) \in G(C(n))$ for $i\leq n$,
    and $(\zeta(\sE_n, G))^i = 0$  for $i \geq n+1$. 
\item  for each $i \in \Z$, there exists $u_i \in G(J(C))$ such that $u_i = c_i (\sE_n, G) $ for all $n\geq 2g-1$ and $0 \leq i \leq n- g + 1$, and $u_i = 0$ for all $i < 0$ and $i > g$.  
\end{enumerate}
Let $\alpha = \sum_{0 \leq i \leq g} u_i z^{g-i} \in G(J(C))[z]$ and let 
\[
I_n  = \begin{cases}  ((\alpha): z^{2g-1-n}) & \textnormal{ if } n <  2g-1,\\
(\alpha \cdot z^{n-2g+1})  &\textnormal{ if } n \geq  2g-1.
\end{cases}
\]
For every $n\geq 0$, we then have the following short exact sequence of $G(J(C))$-modules.
\begin{equation} \label{eqn:cor-to-A*-1}
0 \to I_n \to G(J(C))[z] \xrightarrow{\psi_n} G(C(n)) \to 0, 
\end{equation}
where $\psi_n$ maps $z$ to $(\iota_{n-1,n})_* (1) \in G(C(n))$ if $n \geq 1$, and $0$ otherwise.  
 \end{thm}
 
 \vskip .2cm

\subsection{Theorem~B}\label{sec:thm-B}

In this section, we focus on the invariants which are modules over the  Chow rings and may not have a  ring structure. Let $G\colon  \ProjSm_k^{\op} \to \Ab$ be a presheaf of abelian groups on smooth
 projective  schemes that satisfies {\bf P1-P4} of \secref{sec:Mot-Func} and 

\begin{description}
\item[P5]  Let $\sE \to X$ be a vector bundle of rank $r+1$ over $X$  and let $c_i(\sE) \in \CH^*(X)$ be the $i$-th Chern class of $\sE$. Let $\zeta =  \sO_{\P(\sE)}(1)$ denote the canonical line bundle on $\pi\colon  \P(\sE) \to X$. Then $G$ satisfies the projective bundle formula, i.e., the following sequence of $\CH^*(X)$-modules is exact. 
 \begin{equation} \label{eqn:P-B-F-CH-M*-1}
 0 \to (\sum_i c_i (\sE) z^{r+1-i} ) G(X)[z] \to G(X)[z] \xrightarrow{\phi} G(\P(\sE)) \to 0,
 \end{equation}
 where $
 \phi(x_0 + x_1 z + \cdots + x_n z^n ) = \pi^*(x_0) + c_1(\zeta) \cdot  \pi^*(x_1) + \cdots + c_1(\zeta)^n \cdot  \pi^*(x_n)$. Here $c_1(\zeta) \in \CH^*(\P(\sE))$ is the first Chern class of  the line bundle $\zeta = \sO_{\P(\sE)}(1)$. Note that since $G(X)$ is a $\CH^*(X)$-module, the group $G(X)[z] =  \CH^*(X)[z] \otimes_{\CH^*(X)}  G(X) $ is a $\CH^*(X)[z]$-module and therefore $(\sum_i c_i (\sE) z^i )G(X)[z] \subset G(X)[z]$ is a well defined submodule. 
\end{description}

  \begin{thmA} \label{thm:thm-B}
 Let $C$ be a smooth projective curve over an algebraically closed field $k$ and let $g$ denote the genus of $C$. 
 Let $G\colon  \ProjSm_k^{\op} \to {\bf Ab}$ satisfy  {\bf P1-P5}. For each $n\geq 2g-1$, let $\sE_n$ be the vector bundle on $J(C)$ such that $C(n) \cong \P(\sE_n)$ and let 
 $u_i = c_i(\sE_n) \in \CH^i(J(C))$ for $0\leq i \leq n-g+1$.  By \eqref{eqn:Chern-C*-1}, it follows that $u_i$ does not depend on $n$. 
Let $\alpha = \sum_{0 \leq i \leq g} u_i z^{g-i} \in \CH^*(J(C))[z]$, $M= G(J(C))[z]$ and 
\[
N_n  = \begin{cases}  ((\alpha): z^{2g-1-n}) M &\textnormal{ if } n <  2g-1,\\
(\alpha \cdot z^{n-2g+1})M  &\textnormal{ if } n \geq  2g-1.
\end{cases}
\]
For every $n\geq 0$, we then have the following short exact sequence of $\CH^*(J(C))$-modules.
\begin{equation} \label{eqn:thm-B*-1}
0 \to N_n \to G(J(C))[z] \xrightarrow{\psi_n} G(C(n)) \to  0,
\end{equation}
where for $i \geq 1$ and $x \in G(J(C))$, we have $\psi_n( x z^i)= \left( (\iota_{n-1,n})_* (1) \right)^i \cdot \pi_n^*(x) \in G(C(n))$ if $n \geq 1$, and $0$ otherwise.  
 \end{thmA}
 
 \begin{proof}
 By \lemref{lem:Ext*-1}, $G$ extends to an additive functor $F\colon  \sM_k^{0, \dim} \to \Ab$. 
 In particular, for all $n \geq m$, \lemref{lem:inj} yields that the push-forward map $(\iota_{m,n})_*\colon  G(C(m)) \inj G(C(n))$ is injective and \lemref{lem:surj} implies that the pull-back map $\iota_{m,n}^*\colon  G(C(n)) \to G(C(m))$ is surjective. Let 
   $v_j = (\iota_{j-1, j})_* (1) \in \CH^*(C(j))$ for all $j \geq 1$, and $v_0 = 0$. 
   By \cite[Lemma~8]{Co}, we have 
      \begin{equation}\label{eqn:thm-B*-1.5}
   (\iota_{m, n})_* (1) = v_{n}^{n-m}.
   \end{equation}
It follows from \cite[Lemma~9]{Co} and \cite{Mat} that for all $n\geq 2g-1$, $m \leq n$ and $y \in G(C(n))$, we have 
\begin{equation}\label{eqn:thm-B*-2}
\iota_{m, n}^* (c_1(\sO_{\P(\sE_n)}(1)) \cdot y) = v_m \cdot \iota_{m, n}^*(y).
\end{equation} 
 Indeed, by \cite{Mat}, it follows that for all $n \geq 2g-1$, we have $c_1( \sO_{\P(\sE_n)}(1)) = (\iota_{n-1, n})_* (1)$. The claim then follows from \cite[Lemma~9]{Co} because $\iota_{m,n}^*\colon G(C(n)) \to G(C(m))$ is a homomorphism of $\CH^*(C(n))$-modules.

If $n\geq 2g-1$, then the theorem follows from the projective bundle formula for $\sE_n$. Indeed, if $n \geq 2g-1$, then ${\rm rank}(\sE_n) = n-g+1$ and $c_i(\sE_n) = u_i$ for all $0 \leq i \leq n- g + 1$. The short exact sequence \eqref{eqn:thm-B*-1} now follows from \eqref{eqn:P-B-F-CH-M*-1} because $c_1( \sO_{\P(\sE_n)}(1)) = (\iota_{n-1, n})_* (1)$.

If $1\leq n < 2g-1$, then we define a  homomorphism $\psi_n\colon  G(J(C))[z] \to G(C(n))$ such that the following diagram commutes. 
\begin{equation}\label{eqn:thm-B*-4}
\xymatrix@C3pc{
G(J(C))[z] \ar@{->>}[r]^{\psi_{2g-1}} \ar[rd]_{\psi_n}&G(C(2g-1)) \ar@{->>}[d]^{\iota_{n, 2g-1}^*} \\
& G(C(n)).}
\end{equation}
Observe that $\psi_{2g-1}$ is surjective by the projective bundle formula for $\sE_{2g-1}$ and the right vertical arrow ${\iota_{n, 2g-1}^*}$ is surjective by \lemref{lem:surj}. In particular, $\psi_n$ is surjective. For $x \in G(J(C))$ and $i \geq 0$, the commutative  diagram \eqref{eqn:thm-B*-4} along with  \eqref{eqn:thm-B*-2} yields that 
$\psi_n(x z^i) = {\iota_{n, 2g-1}^*} (\psi_{2g-1} (x z^i)) = {\iota_{n, 2g-1}^*} (\zeta(\sE_{2g-1})^i \cdot \pi_{2g-1}^* (x))  = ((\iota_{n-1, n})_* (1))^i \cdot \pi_n^*(x)  \in G(C(n))$. In short, we have
\begin{equation}\label{eqn:thm-B*-5}
\psi_n(x z^i) = v_n^i \cdot \pi_n^*(x). 
\end{equation}
 It now suffices to  show that $N_n = \ker(\psi_n)$. Let $a = \sum_j x_j z^j  \in G(J(C))[z]$. 
 Then 
 \begin{eqnarray*}
  0 = \psi_n(a) &\iff& 0 =  {\iota_{n, 2g-1}^*} \circ \psi_{2g-1} (a)\\
  & \iff^1 &  0 =  (\iota_{n, 2g-1})_* (1 \cdot \iota_{n, 2g-1}^* ( \psi_{2g-1} (a)))\\
    & \iff^2 & 0 = (\iota_{n, 2g-1})_*(1) \cdot \psi_{2g-1}(a)  \\
      & \iff^3 & 0 = v_{2g-1}^{2g-1-n} \cdot \psi_{2g-1}(a) \\
    & \iff ^4 & 0 = \psi_{2g-1}(a \cdot z^{2g-1-n})\\
    & \iff & a \cdot z^{2g-1-n} \in \ker(\psi_{2g-1}) = N_{2g-1} = (\alpha)M\\
    & \iff & a \in ((\alpha)M: z^{2g-1-n}),
 \end{eqnarray*}
 where $\iff^1$ follows  from \lemref{lem:inj}, $\iff^2$ follows from the projection formula \eqref{eqn:Assum-P-F}, 
 $\iff^3$ follows from \eqref{eqn:thm-B*-1.5} and $\iff^4$ follows from \eqref{eqn:thm-B*-4}. This completes the proof for $n \geq 1$.  Note that $n=0$ case can be proven on the similar lines with the convention ${C(-1)} = \emptyset$ and $(\iota_{-1, 0})_* = 0$.
 \end{proof}
 
 \vskip .2cm

 \section{Applications}\label{sec:App}
 
 In this section, we prove Theorems~\ref{thm:WCT-Structure}--\ref{thm:K-thy-Structure} 
as applications of  Theorems~\ref{thm:thm-A} and~\ref{thm:thm-B}. We recall the following notations 
from Sections~\ref{sec:Symm-P-Corr-On-Symm} and~\ref{sec:Symm-P-Corr-J}.
 
 Let $C$ be a smooth projective curve over an algebraically closed field $k$ and let $g$ be the genus of $C$. Let $J(C)$ denote the Jacobian of $C$. 
 Let $p \in C$ be a closed point and let $\pi_n\colon C(n) \to J(C)$ denote the map defined by the point $p$.
 For each $n \geq 2g-1$, let $\sE_n \to J(C)$  be the vector bundle over $J(C)$ such that $\P(\sE_n) \cong C(n)$. For $m\leq n$, let $\iota_{m, n}\colon C(m) \inj C(n)$ be the closed embedding defined by the closed point $p$.   
 
 \vskip .2cm
 
 \subsection{Classical Weil cohomology theories}\label{sec:App-WCT}
 
 Let $k$ be an algebraically closed field and let $K$ be a field of characteristic zero. Let $H^*\colon  \ProjSm_k \to {\bf GrAlg}_K$ be a classical Weil cohomology theory. For the definition and properties of a Weil cohomology theory, the reader can refer to \cite{Kle68}. Recall that for $X \in \ProjSm_k$, there exists a cycle class map ${\rm cl}: \CH^*(X) \to H^{2*}(X)$ which is a homomorphism of graded rings.

 \begin{prop} \label{cor:inj-surj-WCT}
 Let $H^*\colon  \ProjSm_k \to {\bf GrAlg}_K$ be a classical Weil cohomology theory as above. Then  the push-forward map 
$(\iota_{m, n})_*\colon H^*(C(m)) \to H^*(C(n))[-2(n-m)]$ is injective and the pull-back map 
$(\iota_{m, n})^*\colon H^*(C(n)) \to H^*(C(m))$ is surjective for all $n \geq m$.
 \end{prop}
 \begin{proof}
 It follows from \cite[\S~1.3]{Kle68} that  the Weil cohomology $H^*$ extends to an additive functor 
 $F\colon  \sM_k \to {\bf GrVec}_K$ such that $F((X, 1, m)) = H^*(X)\otimes_K (H^2(\P^1)[-2])^m$. Fixing an isomorphism $K \xrightarrow{\cong} H^2(\P^1)$, we get an isomorphism 
 $F((X, 1, m)) \cong H^*(X)[-2m]$. The proposition  now follows from Lemmas \ref{lem:inj} and \ref{lem:surj} (see \remref{rem:Gr-Fun}).
 \end{proof}
 
 Note that we have the projective bundle formula \eqref{eqn:PBF-C-M} in $\sM_k$. We next prove that this formula with values in abelian groups implies the projective bundle formula  \eqref{eqn:P-B-F} with values in rings (and in this case, with values in graded $K$-algebras).

 \begin{lem}\label{lem:PBF-WCT-Ring}
  Let $H^*$ be a classical Weil cohomology theory. Then $H^*\colon  \ProjSm_k \to {\bf GrAlg}_K$ satisfies  
 the projective bundle formula {\bf PBF} of \secref{sec:thm-A}. Moreover, for a vector bundle $\sE\to X$ and  $i\geq 0$, the element $c_i(\sE, H^*) \in H^*(X)$ appearing in \eqref{eqn:P-B-F} is the image of the $i$-th Chern class $c_i(\sE) \in \CH^*(X)$, and $\zeta(\sE, H^*) = c_1(\sO_{\P(\sE)})$. 
 \end{lem}
 
 \begin{proof}
  Let $X$ be a smooth projective scheme over $k$ and let $\pi\colon  \sE \to X$ be a vector bundle of rank $e+1$. The projective bundle formula \eqref{eqn:PBF-C-M} then implies that 
 we have a surjective map  $\psi\colon   H^{*}(X)[z] \surj  H^*(\P(\sE)) $ defined by 
  $\psi(\sum_j a_j z^j) = \sum_j \eta^j \cdot \pi^*(a_j) $, where $\eta = c_1(\sO_{\P(\sE)}) \in H^2(\P(\sE))$. 
  Observe that $\psi$ is an $H^*(X)$-algebra homomorphism. Let $\beta^{\prime} =  \sum_{0 \leq i \leq e+1} c_i(\sE) z^{e+1-i} \in H^*(X)[z]$. Observe that  $c_i(\sE)$ and $\eta$ are the images of the 
  corresponding elements in $\CH^*(X)$ and $\CH^*(\P(\sE))$, respectively. 
 By the projective bundle formula for $\CH^*(-)$, it follows that $\psi(\beta^{\prime})=0$. 
  Since $c_{0}(\sE)= 1$, $\beta^{\prime}$ is a monic polynomial in $z$. 
  Let 
    $0 \neq \gamma = \sum_{j=0}^t x_j z^j \in  H^{*}(X)[z]$ such that $\psi(\gamma)= 0$. By 
    \eqref{eqn:PBF-C-M}, we have $t \geq e+1$. If $t = e+1$, then $\gamma- x_t \beta^{\prime} \in \sum_{i =0}^e H^*(X)z^i$ and   $\psi(\gamma- x_t \beta^{\prime}) =0$. By
    \eqref{eqn:PBF-C-M}, it follows that $\gamma- x_t \beta^{\prime} =0$, i.e., $\gamma \in (\beta^{\prime})H^*(X)[z]$. If $t> e+1$, then $\gamma- x_t \beta^{\prime} z^{t-e+1} \in \sum_{i =0}^{t-1} H^*(X)z^i$ and   $\psi(\gamma- x_t \beta^{\prime} z^{t-e+1}) =0$. By induction on $t\geq e+1$, we conclude that 
    $\ker(\psi) = (\beta^{\prime})H^*(X)[z]$. This proves \eqref{eqn:P-B-F} and completes the proof. 
 \end{proof}
 
 \subsection{Proof of \thmref{thm:WCT-Structure}}
 It follows from \cite[\S~1.2]{Kle68} that the functor $H^*$ satisfies the properties {\bf P1-P4} of \secref{sec:Mot-Func} and hence extends to an additive functor $F\colon  \sM_k^{0, \deg} \to \Ab$ such that $F((X, 1, 0))= H^*(X) = F((X, 1, \dim(X)))$. 
 Observe that $H^*$ actually extends to a $\Q$-linear functor 
  $F\colon  \sM_k(\Q) \to {\bf GrVec}_K$ such that 
 $F((X, 1, i)) = H^*(X)[-2i]$. Composing this with the forgetful functor  ${\bf GrVec}_K \to \Ab$, the resulting functor gives the desired extension  $F\colon  \sM_k \to \Ab$. In particular,  $F((X, 1, i)) =  H^*(X)$, as abelian groups and $H^*$ satisfies {\bf Extn}  \eqref{eqn:general*-1}.
  By \lemref{lem:PBF-WCT-Ring}, $H^*$ satisfies the projective bundle formula {\bf PBF} \eqref{eqn:P-B-F} and it has the projection formula  {\bf PF} \eqref{eqn:P-F} by  \cite[\S~1.2]{Kle68}.
 
In terms of the notations from \thmref{cor:thm-A},
by \eqref{eqn:Con-Chern-C} and \cite[Lemma~8]{Co}, it follows that for $n \geq 2g-1$, we have 
$\zeta(\sE_n, H^*)^i = (c_1(\sO(1)_{C(n)}))^i = ((\iota_{n-1, n})_* (1))^i = (\iota_{n-i, n})_* (1) \in H^*(C(n))$. 
By \eqref{eqn:Chern-C*-1} and \lemref{lem:PBF-WCT-Ring}, there exist $u_i \in H^*(J(C))$ such that $c_i(\sE_n, H^*) = u_i = {\rm cl}(c_i(\sE_n))$ for all $n\geq 2g-1$ and 
$0 \leq i \leq n- g + 1$, and $u_i = 0$ for all $i < 0$ and $i>g$. The theorem now follows from \thmref{cor:thm-A}. 
$\hfill \square$

\vskip .2cm

 \subsection{Higher Chow groups}\label{sec:App-HCG}

Given a smooth quasi-projective scheme $X$ over a field $k$ and $r \geq 0$,  Bloch \cite{Bl-86} defined the higher Chow groups $\CH^*(X, r)$ such that $\CH^*(X, 0)$ is the Chow ring. 
For the definition and other basic properties of these groups, the reader should refer to \cite{Bl-86} or \cite{Le-94}.

 \begin{lem}\label{lem:PBF-HCG-Ring}
  The functor  $\CH^*(-, \cdot)\colon  \ProjSm_k \to {\bf GrAlg}_K$ satisfies the projective bundle formula \eqref{eqn:P-B-F} such that for a vector bundle $\sE \to X$ and $i\geq 0$, 
  $c_i(\sE, \CH^*(-, \cdot)) = c_i(\sE) \in \CH^*(X, 0)$ and $\zeta(\sE, \CH^*(-,\cdot)) = c_1(\sO_{\P(\sE)}(1))$.
 \end{lem} 
 \begin{proof}
 Let $X$ be a smooth quasi-projective scheme over $k$. 
  For a vector bundle $\pi\colon  \sE \to X$ of rank $e+1$ and $p, r \geq 0$, \cite[Property (vi), p.269]{Bl-86} (or  \cite[Corollary~5.4]{Le-94}) yields an isomorphism 
  \begin{equation} \label{eqn:PBF-HCG}
  \psi\colon  \oplus_{i = 0}^{e} \CH^{p-i}(X, r) \xrightarrow{\cong} \CH^p(\P(\sE), r) 
  \end{equation}
  such that 
  $\psi(a_0, \dots, a_e) = \sum_i \pi^*(a_i) \eta^i$, where $\eta$ is the first Chern class of the canonical line bundle on $\P(\sE)$. 
    The projective bundle formula \eqref{eqn:PBF-HCG} then implies that 
 we have a surjective map  $\psi(r)\colon   \CH^{*}(X, r)[z] \surj  \CH^*(\P(\sE), r) $ defined by 
  $\psi(r)(\sum_j a_j z^j) = \sum_j \pi^*(a_j) \cdot \eta^j$. Observe that $\psi(r)$ is a $\CH^*(X)$-module homomorphism. Let $\beta =  \sum_{0 \leq i \leq e+1} c_i(\sE) z^{e+1-i} \in \CH^*(X)[z]$. By the projective bundle formula for $\CH^*(-)$, we know that $\psi(0)(\beta)=0$. Since $c_{0}(\sE)= 1$, $\beta$ is a monic polynomial in $z$. 
  Let 
    $0 \neq \gamma = \sum_{j=1}^t x_j z^j \in  \CH^{*}(X, r)[z]$ such that $\psi(r)(\gamma)= 0$. By 
    \eqref{eqn:PBF-HCG}, we have $t \geq e+1$. If $t = e+1$, then $\gamma- x_t \beta \in \sum_{i =0}^e \CH^*(X, r)z^i$ and   $\psi(r)(\gamma- x_t \beta) =0$. By
    \eqref{eqn:PBF-HCG}, it follows that $\gamma- x_t \beta =0$, i.e., $\gamma \in (\beta)\CH^*(X, r)[z]$. If $t> e+1$, then $\gamma- x_t \beta z^{t-e+1} \in \sum_{i =0}^{t-1} \CH^*(X, r)z^i$ and   $\psi(r)(\gamma- x_t \beta z^{t-e+1}) =0$. By induction on $t\geq e+1$, we conclude that 
    \begin{equation} \label{eqn:PBF-r-level}
    \ker(\psi(r)) = (\beta)\CH^*(X, r)[z].
    \end{equation} 
    Taking direct sum over $r\geq 0$, we obtain the short exact sequence 
    \[
    0 \to (\beta)\CH^*(X, \cdot)[z] \to \CH^*(X, \cdot)[z] \xrightarrow{\psi} \CH^*(\P(\sE), \cdot)\to 0.
    \]
    This proves the projective bundle formula \eqref{eqn:P-B-F} for $G= \CH^*(-, \cdot)$ and completes the proof of the lemma. 
 \end{proof}

We now generalise the results of \cite{Co} for the functor $\CH^*(-, \cdot)$.  The first part of the following proposition  was stated in \cite{K-I}.

\begin{prop} \label{cor:inj-surj-HCG}
The push-forward map 
$(\iota_{m, n})_*\colon \CH^*(C(m), r) \to \CH^*(C(n), r)$ is injective and the pull-back map 
$(\iota_{m, n})^*\colon \CH^*(C(n), r) \to \CH^*(C(m), r)$ is surjective 
for all $m \leq n$ and $r\geq 0$.
\end{prop}
\begin{proof}
It is well known that $\CH^*(-, r)\colon \Sm_k \to \Ab$ extends to an additive functor $\CH^*(-, r)\colon  \Co_k \to \Ab$. For example, with rational coefficients, it follows from \cite[Corollary~5.3]{Le-94}. The extension also follows from  \lemref{lem:Ext*-1} and \cite[p.268-269, Corollary~5.7]{Bl-86}. The proposition then follows from Lemmas \ref{lem:inj} and \ref{lem:surj}.
\end{proof}

For $X\in \ProjSm_k$, we shall consider the bi-graded ring 
\[
\CH^*(X, \cdot)[z] = \oplus_{(i, j) \in \Z_{\geq 0} \times \Z_{\geq 0}} \left( \oplus_{0\leq s \leq i} \CH^{s}(X, j) z^{i-s}        \right).
\]

\subsection{Proof of \thmref{thm:HCG-Structure}}
 By \cite[p.268-269, Corollary~5.7, Exercise~5.8(i)]{Bl-86} and  \lemref{lem:PBF-HCG-Ring}, it follows that the functor $\CH^*(-, \cdot)$ has the  properties {\bf PBF}, {\bf Extn} and {\bf PF} of \secref{sec:thm-A}.
 We now verify the remaining hypothesis of \thmref{cor:thm-A} for the functor $G(-) = \CH^*(-, \cdot)$.  In terms of the notations from \thmref{cor:thm-A},
by \eqref{eqn:Con-Chern-C} and \cite[Lemma~8]{Co}, it follows that 
$\zeta(\sE_n)^i = (c_1(\sO(1)_{C(n)}))^i = ((\iota_{n-1, n})_* (1))^i = (\iota_{n-i, n})_* (1) \in \CH^*(C(n), \cdot)$
for $n \geq 2g-1$. By \eqref{eqn:Chern-C*-1} and  \lemref{lem:PBF-HCG-Ring}, there exist $u_i \in \CH^*(J(C), \cdot)$ such that $c_i(\sE_n, \CH^*(-, \cdot)) = u_i = c_i(\sE_n)$ for all $n\geq 2g-1$
 and 
$0 \leq i \leq n- g + 1$, and $u_i = 0$ for all $i < 0$ and $i>g$. The theorem now follows from \thmref{cor:thm-A}. 
 $\hfill \square$

\vskip .2cm

\subsection{Additive higher Chow groups}\label{sec:App-AHCG}

For a smooth scheme $X$ and  $r , m \geq 0$, we let $\TCH^*(X, r; m)$ denote the additive higher 
chow group of $X$. For the definition and other basic properties of the additive higher Chow groups, the reader should refer to \cite[\S~2]{Park} and \cite[\S~3]{KLevine}.
Recall that by \cite[Corollary~5.4]{KLevine}, the assignment $X \mapsto \TCH^*(X, r; m)$ defines a functor $\TCH^*(-, r ; m)\colon  \Sm_k^{\op} \to \Ab$. Moreover, by \cite[\S~3.3]{KLevine}, 
it follows that $\TCH^*(-, r; m)$ has projective push-forwards.


%
%
%

\begin{prop} \label{cor:inj-surj-AHCG}
The push-forward map 
$(\iota_{l, n})_*\colon \TCH^*(C(l), r ;m) \to \TCH^*(C(n), r;m)$ is injective and the pull-back map
$
(\iota_{l, n})^*\colon \TCH^*(C(n), r;m) \to \TCH^*(C(l), r; m )
$ is surjective for all $l \leq n$ and $r, m\geq 0$.
\end{prop}

\begin{proof}
By \cite[Theorem~5.3]{KLevine}, it follows that $\TCH^*(-, r;m)$ defines an additive  functor $\TCH^*(-, r ;m)\colon  \sM_k^{0, \dim} \to \Ab$. The result now follows from Lemmas \ref{lem:inj} and \ref{lem:surj}. 
\end{proof}

\subsection{Proof of \thmref{thm:AHCG-Structure}}

By \cite[\S~3.3, Corollary~5.4, Lemma~3.8, Theorem~4.10]{KLevine}, it follows that the functor $G = \TCH^*(- , r , m)$ satisfies {\bf P1-P4} of \secref{sec:Mot-Func}. Moreover, the property {\bf P5} for $G$ follows from \cite[Theorem~5.6]{KLevine} exactly as we proved \eqref{eqn:PBF-r-level} in \lemref{lem:PBF-HCG-Ring}. The theorem then follows from \thmref{thm:thm-B}.
$ \hfill \square$

\begin{remk} \label{rem:HCG-M}
This remark is about the higher Chow groups with modulus. Binda and Saito \cite{BS}  defined
the higher Chow groups with modulus which generalises 
the additive higher Chow groups and
the  higher Chow groups of Bloch. Let $C$ be a smooth projective curve over $k$ and let $D$ be a divisor on $J(C)$. 
If we work with the categories $\Sm_{J(C)}$, $\sM_{J(C)}$ and $\Co_{J(C)}$, then we can obtain results similar to \thmref{thm:AHCG-Structure} for the higher 
Chow groups with modulus $\CH^*(J(C)|D, r)$ and $\CH^*(C(n)|\pi^*(D), r)$.  
This follows from \cite[Theorems~4.3,~4.5 and~4.6]{KP} exactly as we proved \thmref{thm:AHCG-Structure}. The observation we need to make everything work is that all morphisms and correspondences are over $J(C)$. 
\end{remk}

\vskip .2cm

 \subsection{Rational $K$-theory} \label{sec:App-K-thy}
 For $X \in \ProjSm_k$, let $K(X)$ denote the algebraic $K$-theory spectrum and let $K_r(X)$ denote the $r$-th stable homotopy group of $K(X)$. Grothendieck proved that the Chern character defines an isomorphism 
  ${\rm ch}\colon  K_0(X)_{\Q} \xrightarrow{\cong} \CH^*(X)_{\Q}$. Since each $K_n(X)$ is a $K_0(X)$-module, it follows that $K_r(X)_{\Q}$ is a $\CH^*(X)_{\Q}$-module for all $r\geq 0$. 
  
%
%
%
%
%
%
%
%
 
\begin{prop} \label{cor:inj-surj-K-thy}
The push-forward map 
$(\iota_{m, n})_*\colon K_r(C(m))_{\Q} \to K_r(C(n))_{\Q}$ is injective and the pull-back map 
$(\iota_{m, n})^*\colon K_r(C(n))_{\Q} \to K_r(C(m))_{\Q}$ is surjective for all
$m \leq n$ and $r\geq 0$.  
\end{prop}

\begin{proof} It follows from \cite[3.14, 3.16.4, Propositions~3.17 and~3.18]{TT} that the functor $G = K_r(-)_{\Q}$ satisfies {\bf P1-P4} of \secref{sec:Mot-Func}. \lemref{lem:Ext*-1} then yields that  $K_r(-)_{\Q}$ defines an additive  functor $K_r(-)_{\Q}\colon  \sM_k^{0, \dim} \to \Ab$. The result now follows from Lemmas \ref{lem:inj} and \ref{lem:surj}. 
\end{proof}

\subsection{Proof of \thmref{thm:K-thy-Structure}}
 As before,  it follows from \cite[3.14, 3.16.4, Propositions~3.17 and~3.18]{TT} that $K_r(-)_{\Q}$ satisfies  {\bf P1-P4} of \secref{sec:Mot-Func}.
Moreover, the property {\bf P5} follows from \cite[Theorem~7.3]{TT} because under the isomorphism $K_0(X)_{\Q} \xrightarrow{\cong} \CH^*(X)_{\Q}$, the class of  $\sO_{\P(\sE)}(1)$ maps to $c_1(\sO_{\P(\sE)}(1))$. The theorem now follows from \thmref{thm:thm-B}. 
$\hfill \square$

\begin{remk} 
Observe that Bloch \cite[Theorem~9.1]{Bl-86} proved that for a smooth quasi-projective scheme, the character map ${\rm ch_X}\colon K_n(X)_{\Q} \to \CH^*(X, n)_{\Q}$ is an isomorphism. Indeed, the theorem says that ${\rm Td}(X) {\rm ch}_X$ is an isomorphism but the Todd class ${\rm Td}(X)$ belongs to $\CH^*(X)^{\times}$. Since  the isomorphism ${\rm ch}_X$ commutes with the pull-backs, \thmref{thm:K-thy-Structure} follows from \thmref{thm:HCG-Structure}. Here we are able to avoid such an isomorphism and this helps us in \remref{rem:rel-K-thy}.
\end{remk}

 \begin{remk}\label{rem:rel-K-thy}
 Since the relative $K$-theory also satisfies  properties of {\bf P1-P5}, we have variants of \propref{cor:inj-surj-K-thy} and \thmref{thm:K-thy-Structure} for the relative $K$-theory (with rational coefficients) as well.  One way to generalise these results is that we can fix a pair $(X, Y)$ such that $Y \inj X$ is a closed subscheme of $X$ and study the functor $K_r(-\times X, -\times Y) \colon  \SmProj_k \to \Ab$. This is similar to the case of the additive higher Chow groups, where $(X, Y) = (\A^1, (m+1)\{0\})$. The other way is to fix a closed subscheme of $J(C)$ and work with $\SmProj_{J(C)}$, see \remref{rem:HCG-M}. This case is similar to the case of the higher Chow groups with modulus. 
 \end{remk}

 \vskip .2cm

\noindent\emph{Acknowledgments.}
The author is supported by the SFB 1085 \emph{Higher Invariants} (Universit\"at Regensburg). 
The author would like to thank Dr. Anand Sawant for his useful suggestions on the structure of the 
paper which led to an improved exposition. The author also thanks the referees for reading the manuscript thoroughly and suggesting 
valuable improvements.

\end{document}